\documentclass[11pt]{amsart}
\usepackage{amsmath,amsfonts,amssymb,amsthm,amscd,latexsym,euscript,verbatim,
bbold}
\usepackage[all]{xy}

\makeatletter
\def\section{\@startsection{section}{1}%
  \z@{1.1\linespacing\@plus\linespacing}{.8\linespacing}%
  {\normalfont\Large\scshape\centering}}
\makeatother

\theoremstyle{plain}

\newtheorem*{thmA}{Theorem A}

\newtheorem*{conj*}{Root Groups Conjecture}
\newtheorem*{thm1.2}{(1.2) Theorem}
\newtheorem*{thm1.3}{(1.3) Theorem}
\newtheorem*{thm1.4}{(1.4) Theorem}
\newtheorem*{prop*}{Proposition}
\newtheorem*{thm*}{Theorem}

\newtheorem{prop}{Proposition}[section]

\newtheorem{lemma}[prop]{Lemma}

\theoremstyle{definition}

\newtheorem*{Def*}{Definition}
\newtheorem{Defs}[prop]{Definitions}

\newtheorem*{notation*}{Notation}
\newtheorem{remark}[prop]{Remark}

\newcommand{\ff}{F}

\newcommand{\ga}{\alpha}
\newcommand{\gb}{\beta}
\newcommand{\gc}{\gamma}

\newcommand{\gs}{\sigma}

\newcommand{\charc}{{\rm char}}

\newcommand{\sminus}{\smallsetminus}

\newcommand{\one}{{\bf 1}}

\numberwithin{equation}{section}

\begin{document}
\title[Uniform characterization of the octonions and the quaternions]{A uniform characterization of the octonions and the
quaternions using commutators}
\author[Kleinfeld and Segev]{Erwin Kleinfeld\qquad Yoav Segev}

\address{Erwin Kleinfeld \\
1555 N.~Sierra St.~Apt 120, Reno, NV 89503-1719, USA}
\email{erwinkleinfeld@gmail.com}

\address{Yoav Segev \\
         Department of Mathematics \\
        Ben-Gurion University \\
        Beer-Sheva 84105 \\
         Israel}
\email{yoavs@math.bgu.ac.il}

 
\dedicatory{In honor  of professor Amitsur, for his centennial symposium 2021}

\begin{abstract}
Let $R$ be a ring with $\one,$ which is not commutative.
Assume that a non-zero commutator in $R$ is not a zero divisor.
Assume further that either $R$ is alternative, but not associative,
or $R$ is associative and any commutator $v\in R$ satisfies: $v^2$
is in the center of $R.$ 

We prove that $R$ has no zero divisors.
Furthermore, if $\text{char}(R)\ne 2,$ then the localization of $R$
at its center is an  octonion division algebra, if $R$ is alternative
and a quaternion division algebra, if $R$ is associative.

Our proof in both cases is essentially the same and it is elementary and rather self contained.
\end{abstract}
\date{\today}
 
\maketitle


\section{Introduction}
The new results in this paper are first, that an alternative ring with $\one$ which is not commutative
and not associative,
and which satisfies hypothesis (i) of Theorem A has no zero divisors.
Second, we give a new proof of an old theorem of Bruck and Kleinfeld \cite[Theorem A]{BK},
that if $R$ is an alternative ring without zero divisors and of characteristic not $2,$
then the localization of $R$ at its center
is an octonion divison algebra.  In \cite{BK} associators are used for the proof,
and in this paper we use commutators.

This makes it possible to prove Theorem A, in which the quaternions and the octonions
are dealt with in a uniform way (see Remark \ref{MR} below), with a similar characterization.
The results in Theorem A about the quaternions appear in \cite{KS2}.

\begin{thmA}
Let $R$ be a ring with $\one$ which is not commutative.
Assume that
\begin{itemize}
\item[(i)]
A non-zero  commutator in $R$  is not a divisor of zero in $R,$ and

\item[(ii)]
one of the following holds.
\begin{itemize}
\item[(a)]
$R$ is an alternative ring which is not associative.

\item[(b)]
$R$ is an associative ring such that $(x,y)^2\in C,$ for all $x,y\in R,$ where $C$ is the center of $R.$
\end{itemize}
\end{itemize}
Then
\begin{enumerate}
\item
$R$ contains no divisors of zero.

\item
Suppose, in addition, that the characteristic of $R$ is not $2,$
and let $R//C$ be the localization of $R$ at $C.$ 
If $R$ is alternative, then $R//C$ is an octonion division algebra,
and if $R$ is associative, then $R//C$ is
a quaternion division algebra.
\end{enumerate}
\end{thmA}

We note that if $x,y\in R$ are non-zero elements such that $xy=0,$
then we say that both $x$ and $y$ are zero divisors in $R.$ 
Recall also that the commutator $(x,y)=xy-yx.$

\begin{remark}\label{MR}
Notice that by Lemma \ref{facts}(3(ii)) below,
if $R$ satisfies hypotheses (i) and (ii(a))
of Theorem A,
then $(x,y)^2\in C,$ for all $x,y\in R,$ where $C$ is the center of $R.$
\end{remark}

\section{Preliminaries on alternative rings}

Our main references for alternative rings are \cite{K1, K2}.
Let $R$ be a ring, not necessarily with $\one$ and not
necessarily associative.  

\begin{Defs}
Let $x,y,z\in R.$
\begin{enumerate}
\item
The {\bf associator} $(x,y,z)$ is defined to be
\[
(x,y,z)=(xy)z-x(yz).
\]

\item
The {\bf commutator} $(x,y)$ is defined to be
\[
(x,y)=xy-yx.
\]

\item
$R$ is an {\bf alternative ring} if
\[
(x,y,y)=0=(y,y,x),
\]
for all $x,y\in R.$  It is well known and is a theorem of E.~Artin,
that $R$ is an alternative ring if and only if any subring of $R$ generated by two elements is associative.
This fact will be used throughout this paper.

\item
The {\bf nucleus} of $R$ is denoted $N$ and defined
\[
N=\{n\in R\mid (n,R,R)=0\}.
\]
Note that in an alternative ring the associator is skew symmetric in its $3$
variables (\cite[Lemma 1]{K2}).  Hence   $(R,n,R)=(R,R,n)=0,$ for $n\in N.$

\item
The {\bf center} of $R$ is denoted $C$ and defined
\[
C=\{c\in N\mid (c,R)=0\}.
\]
\end{enumerate}
\end{Defs}

In the remainder of this section {\bf $R$ is an alternative ring}
which is {\bf not associative}. 
$N$ denotes the nucleus of $R$ and $C$ its center.

\begin{prop}\label{facts}
Let $R$ be an alternative ring which is not associative,  and let $v\in R$ be a commutator, then

\begin{enumerate}
\item
$v^4\in N.$ 

\item
$[v^2,R,R]v=0.$

\item
\begin{itemize}
\item[(i)]
If $v$ is not a zero divisor in $R,$ then $v^2\in N.$

\item[(ii)]
If $(x,y)$ is not a zero divisor, for all $x,y\in R,$ then
$N=C,$ so $(x,y)^2\in C,$ for all $x,y\in R.$ 
\end{itemize}
\end{enumerate}
\end{prop}
\begin{proof}
For (1)\&(2) see \cite[Theorem 3.1]{K1}.
Part (3(i)) is an immediate consequence of (2),
and part (3ii) follows from the fact that
$(w,n)(w,n)(x,y,z)=0,$ for all $w.x,y,z\in R,$
and all $n\in N,$ see \cite[Lemma 2.4(5)]{KS1}.
\end{proof}

\section{The proof of Theorem A}

In this section $R$ is as in Theorem A.  So $R$ has $\one,$ it is not commutative, 
it satisfies hypothesis (i) of Theorem A and is one of the two possibilities of (ii) of Theorem A. 
$C$ denotes  the center of $R.$
Notice that by hypothesis (ii(b)) of Theorem A,
and by Lemma \ref{facts}(3(ii)),
\[
(x,y)^2\in C, \text{ for all }x,y\in R.
\]

\begin{lemma}\label{lem quadratic}
Let $x\in R\sminus C,$ and let $v=(x,y)$ be a non-zero commutator.
Then
\begin{enumerate}
\item
 $v+vx$ and $vx$ are non-zero commutators.

\item
$ax^2+bx+c=0,$ for some $a,b,c\in C,$ with $a, c$ non-zero.
\end{enumerate}
\end{lemma}
\begin{proof}
(1)\quad
We have $v+vx=v(\one+x)=(x,y(\one+x)),$  and $vx=(x,yx).$
By hypothesis (i) of Theorem A, these commutators are non-zero.

(2)\quad
Let $\ga:=(v+vx)^2=v^2+v^2x+vxv+(vx)^2.$ Then $\ga\in C.$
We have $\ga x=v^2x^2+(v^2+(vx)^2)x+(vx)^2.$  Letting 
\[
\text{$a:=v^2, b:=v^2+(vx)^2-(v+vx)^2$ and $c:=(vx)^2,$}
\]
we see that $a, b, c\in C,$ and $ax^2+bx+c=0,$ with ${a\ne 0\ne c.}$
\end{proof}

\begin{lemma}\label{C}
$R$ contains no divisors of zero.
\end{lemma}
\begin{proof}
Suppose first that $cr=0,$ for some non-zero $c\in C.$
Let $v:=(x,y)$ be a non-zero commutator.
Then $(vc)r=v(cr)=0,$ but $vc=(x,yc)\ne 0,$ hence $r=0,$
so $c$ is not a divisor of $0.$

Suppose next that $x\in R\sminus C,$  and that
$xy=0,$ for some non-zero $y\in R.$ Then we immediately
get from Lemma \ref{lem quadratic}(2) that $cy=0,$  a contradiction.
\end{proof}

\begin{remark}\label{C field}
In view of Lemma \ref{C}, we can
form the {\it localization of $R$ at $C,$}  $R//C.$ This is the
set  of  all  formal fractions $x/c,\  x\in R,\ c\in C,\ c\ne 0,$ 
with the  obvious definitions: (i)  $x/c=y/d$ if and  only if $dx = cy;$  
(ii)  $(x/c) +  (y/d)= (dx+cy)/(cd);$ (iii)  $(x/c)(y/d)=  (xy)/(cd).$ 
It  easy to check that $r\mapsto r/\one$ is an embedding
of $R$ into $R//C$ and that the center of
$R//C$ is the fraction field of $C.$
Since $(x/c,y/d)= (x,y)/cd$ in $R//C$, we may replace $R$ with 
$R//C$ in Theorem A.
  Thus {\bf from now on
we replace $R$ with $R//C$ and assume that $C$ is a field.  We
also assume that $\charc(C)\ne 2.$} 
\end{remark}

The following technical lemma will be used
to construct an octonion division algebra inside
$R,$ when $R$ is alternative.

\begin{lemma}\label{abc}$ $
Suppose $R$ is alternative.
Let $a,b\in R\sminus\{0\}$ be a pair of  anticommutative elements of $R.$
 
\begin{enumerate}
\item
If $c\in R\sminus\{0\}$ anticommutes with $a$ and $b,$ then
for every permutation $\gs$ of $a,b,c$
\begin{itemize}
\item[(i)]
$\gs(a)\big(\gs(b)\gs(c)\big)={\rm sgn}\, \gs a(bc).$

\item[(ii)]
$\big(\gs(a)\gs(b)\big)\gs(c)={\rm sgn}\, \gs (ab)c.$
\end{itemize}

\item
If $a^2\in C,$ then $a$ anticommutes with $ab.$

\item
Suppose that $c\in R$ anticommutes with $a,b$ and $ab.$ Let $\{x,y,z\}=\{a,b,c\},$ then
\begin{itemize}
\item[(i)]
$x$ anticommutes $yz.$

\item[(ii)]
$(xy)z=-x(yz),$ hence $(xy)z, x(yz)\in \{a(bc), -a(bc)\}.$

\item[(iii)]
If, in addition, $a^2,b^2,c^2\in C,$ 
then $\{a,b,c,ab,ac,bc,a(bc)\}$ is a set of pairwise
anticommutative elements.
 
\end{itemize}
\end{enumerate}
\end{lemma}
\begin{proof}
(1)  This appears in \cite[Lemma 7, p.~134]{K1}.
Since associators in $R$ are skew symmetric, $(a,b,c)+(b,a,c)=0.$ Hence
\[
0=(a,b,c)+(b,a,c)=(ab)c-a(bc)+(ba)c-b(ac)=-a(bc)-b(ac),
\]
so $b(ac)=-a(bc)$ this shows (1), and the proof of (2) is similar.
\medskip

\noindent
(2)  $a(ab)=a^2b=ba^2=(ba)a=-(ab)a.$
\medskip

\noindent
3(i)  
We show that $a$ anticommutes with $bc,$ the proof that $b$
anticommutes with $ac$ is similar.  Using (1), we have
\[
a(bc)=-c(ba)=(ba)c=-(bc)a.
\]
\medskip

\noindent
3(ii) 
By 3(i) and (1), $(xy)z=-z(xy)=z(yx)=-x(yz).$
The last part of 3(ii) follows from (1).
\medskip

\noindent
3(iii)  
First we show that 
\[\tag{$\ga$}
\text{$xy$ anti commutes with $xz.$}
\]
Indeed, by (1) and (2), $(xy)(xz)=-(x(xz))y=-x^2(zy),$ since $x^2\in C.$
Similarly $(xz)(xy)=-x^2(yz).$

Next note that
\[\tag{$\gb$}
 \text{$x$ anticommutes with $x(yz).$}
\]
This follows from (2) and 3(i), since $(yz)^2\in C.$

Finally we show that 
\[\tag{$\gc$}
\text{$xy$ anticommutes with $x(yz).$}
\]
Indeed,  by 3(ii) $(xy)(x(yz))=-(xy)((xy)z)=-(xy)^2z.$
And by 3(ii) and 3(i), $(x(yz))(xy)=-((xy)z)(xy)=(z(xy))(xy)=(xy)^2z,$ because ${(xy)^2\in C.}$

Notice now that 3(iii) follows from (2), 3(i), $(\ga)$, $(\gb)$ and $(\gc)$.
\end{proof}

The next proposition is the main tool in this paper.
It is used to construct a quaternion division algebra $Q$ 
inside $R,$ when $R$ is associative, and to prove that $R=Q.$
It is also used to construct an octonion division algebra $O$
inside $R,$ when $R$ is alternative, and to prove that $R=O.$

\begin{prop}\label{anti}
Let $u_1,u_2,\dots,u_n\in R\sminus C,$ such that $u_i^2\in C,$
and $u_{\ell}u_s=-u_su_{\ell},$ for all distinct $\ell,s.$  
Let 
\[
V:=C+Cu_1+\dots+Cu_n,
\]
be the subspace of $R$ spanned by $\one,u_1,\dots,u_n.$
\begin{enumerate}
\item
If $p\in R$ satisfies
\[\tag{$*$}
pu_{\ell}+u_{\ell}p:=d_{\ell}\in C,\quad\text{for all }\ell\in\{1,\dots, n\},
\]
then the element
\begin{equation}\label{m}
\textstyle{m:=p-\sum_{i=1}^n(d_i/2u_i^2)u_i,}
\end{equation}
satisfies $mu_{\ell}+u_{\ell}m=0,$ for all $\ell\in\{1,\dots,n\}.$

\item
If $R\ne V,$
then there exists $p\in R\sminus V$ such that the element $m$ of equation \eqref{m}
satisfies $mu_{\ell}+u_{\ell}m=0,$ for all $\ell\in\{1,\dots,n\},$ and $m^2\in C.$
\end{enumerate}
\end{prop}
\begin{proof}
(1)
We show that $mu_1+u_1m=0,$ the proof for $u_2,\dots,u_n$
is identical.
\[
mu_1+u_1m=
pu_1+u_1p-d_1-\sum_{i=2}^n(d_i/2u_i^2)(u_iu_1+u_1u_i)=0
\]
\medskip

\noindent
(2)
We show that there exists $p\in R\sminus V,$ such that
$p^2\in C,$ and such that $p$ satisfies $(*).$

Let $x\in R\sminus V.$ By Lemma \ref{lem quadratic}, $x$ satisfies a quadratic,
and hence a monic quadratic
equation $x^2-bx+c=0,$ over $C.$ Let $p:=x-b/2.$ Then $p\notin V,$
and $p^2\in C.$  Let $u\in\{u_1,\dots u_n\}.$ Then both $p+u$ and $p-u$ satisfy
a monic quadratic equation over $C.$  That is
\[
\begin{aligned}
&(p+u)^2=c_1(p+u)+c_2\\
&(p-u)^2=c_3(p-u)+c_4.
\end{aligned}
\]
Adding we get
\[
(c_1+c_3)p+(c_1-c_3)u+c_5=0,\quad\text{where }c_5=c_2+c_4-2p^2-2u^2\in C.
\]
Now $c_1+c_3=0,$ since $p\notin V,$  and then $c_1-c_3=0,$ since $u\notin C.$
We thus get that
\[
pu+up=c_2-p^2-u^2\in C.
\]

Let now $m$ be as in equation \eqref{m}.  For $i\in\{1,\dots,n\},$ set ${\ga_i:=(d_i/2u_i^2)\in C.}$  Note that
\[
m^2=p^2+\sum_{i=1}^m\ga_i^2u_i^2+\sum_{i=1}^m \ga_i(pu_i+u_ip)\in C.
\]
\end{proof}

Now we construct that quaternions and the octonions inside $R$ in the respective cases.

\begin{prop}\label{Q,O}$ $
\begin{enumerate}
\item
$R$ contains a quaternion division algebra.
\[
Q=C\one+Ca+Cb+Cab,\qquad a,b\in R\sminus\{0\}, a^2, b^2\in C.
\] 

\item
If $R$ is alternative, then there exists $c\in R\sminus\{0\}$ that anticommutes 
with $a,b$ and $ab$ above, and such that $c^2\in C.$  Hence $R$ contains an octonion division algebra
\[
O:=\ff\one+Cu_1+Cu_2+Cu_3+Cu_4+Cu_5+Cu_6+Cu_7,
\]
where
\[
u_1:=a, u_2:=b, u_3=ab, u_4:=c, u_5=ac, u_6:=bc, u_7:=(bc)a,
\]
so $u_{\ell}\ne 0, u_{\ell}^2\in C,$ and $u_{\ell}u_s=-u_su_{\ell},$ for all $\ell, s.$
\end{enumerate}
\end{prop}
\begin{proof}
(1)  Let $a\in R$ be a nonzero commutator.  Thus $a^2\in C\sminus\{0\}.$
Of course $R\ne \ff \one +\ff a,$ because $R$ is not commutative.
By Proposition \ref{anti}(2), there exists $b\in R\sminus\{0\}$ such that $ab=-ba$ and $b^2\in C.$
Thus $Q:=C\one+Ca+Cb+Cab$ is a quaternion algebra.  Since $R$ contains no divisors
of zero, $Q$ is a division algebra.
\medskip

\noindent
(2) 
By (1) there exist nonzero $a,b\in R$ such that $ab=-ba$ and $a^2, b^2\in C.$
Since $R$ is not associative $R\ne Q,$ where $Q$ is as in (1).  Hence by Proposition \ref{anti}(2),
and since $a,b,ab$ pairwise anticommute, there exists $c\in R\sminus Q,$ such that $c$ anticommute
with $a,b,ab,$ and $c^2\in C.$ By Lemma \ref{abc}(3(iii)), $u_1,\dots,u_7$ satisfy
the assertion of part (2).

Set
\[
\ga:=a^2, \gb:=b^2, \gc=c^2.
\]
It is easy to check now that $\{u_1,\dots,u_7\}$ satisfy the multiplication table
on p.~137 of \cite{K2} (with $\one=u_0).$  Hence $O$
is an octonion algebra.  Since $R$ has no zero divisors, $O$ is a division algebra
(see \cite[section III]{Sc}).
\end{proof}

Lemma \ref{oct} below will be used to show that $R=O,$
when $R$ is alternative, where $O$ is as in Lemma \ref{Q,O}(2) above.

\begin{lemma}\label{oct}
Let $m, a, b, c\in R\sminus\{0\},$ and let $\{x,y,z\}=\{a,b,c\}.$
Assume that
\begin{itemize}
\item[(i)]
$x$ anticommutes with $y$ and $yz.$

\item[(ii)]
$m$ anticommutes with $x, xy, a(bc).$
\end{itemize}
Then
\begin{enumerate}
\item
$mx$ anticommutes with $y, yz;$ $m(xy)$ anticommutes with $z,$
and $(mx)y$ anticommutes with $z.$

\item
$2m(x(yz))=0.$
\end{enumerate}
\end{lemma}
\begin{proof}
(1)  Using Lemma \ref{abc}(1) we get,
\[
(mx)(yz)=-((yz)x)m=m((yz)x)=-(yz)(mx).
\]
and 
\[
(m(xy))z=-(z(xy))m=m(z(xy))=-z(m(xy)).
\]
Also
\[
((mx)y)z=-(zy)(mx)=(mx)(zy)=-z((mx)y).
\]
\medskip

\noindent
(2)  We can now use Lemma \ref{abc}(3(ii)) with $\{m, x,yz\},$ $\{mx, y,z\}$ and  $\{m, xy, z\},$
in place of $\{a,b,c\}.$ Thus we have
\[
m(x(yz))=-(mx)(yz)=((mx)y)z,
\]
and
\[
m(x(yz))=-m((xy)z)=(m(xy))z=-((mx)y)z.\qedhere
\]
\end{proof}

\noindent
{\bf Proof of Theorem A.}\\

\noindent
We can now complete the proof of Theorem A.  So Let $R$ be as in Theorem A.  By 
Lemma \ref{C} $R$ contains no zero divisors.  We now assume that the characteristic of $R$
is not $2,$ and we replace $R$ with $R//C$ as in Remark \ref{C field}.

Let $a, b\in R$ and $Q$ be as in Proposition \ref{Q,O}(1). 
Suppose that $R$ is associative and that $R\ne Q.$  By Proposition \ref{anti}(2),
there exists $m\in R\sminus Q$ that anticommutes with $a,b, ab$.  But then
\[
m(ab)=(ma)b=-(am)b=-a(mb)=(ab)m=-m(ab).
\]
So $2m(ab)=0,$ hence $m(ab)=0.$  But $R$ has no divisors of $0,$ a contradiction.

Suppose now that $R$ is alternative.  Let $O$ be as in Proposition \ref{Q,O}(2).
Assume that $R\ne O.$  By Proposition \ref{anti}(2), there exists $m\in R\sminus O$ that 
anticommutes with $u_1,\dots u_7.$  By Lemma \ref{oct},
$2m(a(bc))=0,$ again a contradiction.\qedhere

\end{document}